\documentclass{amsart}
\usepackage{amsmath,amsfonts,amssymb,amsthm, enumerate,latexsym,mathrsfs, color}
\usepackage[alphabetic]{amsrefs}
\usepackage{graphicx}
\usepackage[all]{xy}
\newcounter{nthm}

\newcounter{nthmm}

\newtheorem{thm}{Theorem}

\newtheorem{thm*}[nthm]{Theorem}
\newtheorem{thm**}[nthmm]{Th\'eor\`eme}

\newtheorem*{defn*}{Definition}
\newtheorem{prop}{Proposition}
\newtheorem{prop*}[nthm]{Proposition}

\newtheorem{lemma}{Lemma}

\newtheorem{question}{Question}

\newcommand{\N}{\mathbb{N}}

\newcommand{\Q}{\mathbb{Q}}

\newcommand{\m}[1]{\textbf{#1}}

\newcommand{\LO}{\mathrm{LO}}

\newcommand{\Aut}{\mathrm{Aut}}

\newcommand{\Age}{\mathrm{Age}}

\newcommand{\actson}{\curvearrowright}

\author{Lionel Nguyen Van Th\'e}

\address{Aix Marseille Univ, CNRS, Centrale Marseille, I2M UMR 7373, 13453 Marseille, France}

\email{lionel.nguyen-van-the@univ-amu.fr}

\thanks{This work has been partially supported by the GrupoLoco project (ANR-11-JS01-0008) funded by the French Government, managed by the French National Research Agency (ANR)}



\date{June 2018}

\title[Glasner's problem]{Glasner's problem for Polish groups with metrizable universal minimal flow}

\begin{document}
\maketitle

\begin{abstract}
A problem of Glasner, now known as Glasner's problem, asks whether there exists a minimally almost periodic, monothetic, Polish group that is not extremely amenable. The purpose of this short note is to observe that a negative answer is obtained under the additional assumption that the universal minimal flow is metrizable.    
\end{abstract}

\section{Introduction}

In \cite{Glasner1998}, Eli Glasner asked whether there exists a minimally almost periodic, monothetic, Polish group that is not extremely amenable. (Recall that a topological group $G$ is \emph{minimally almost periodic} when it admits no non-trivial continuous character, \emph{monothetic} when it contains a dense cyclic subgroup, and \emph{extremely amenable} when every continuous $G$-action on a compact Hausdorff space has a fixed point.) While many experts in the field, including Glasner himself, seem to be convinced that such a group does exist, this problem  - now known as \emph{Glasner's problem} - is still largely open in general (a detailed account can be found in Pestov's contribution in \cite{Pearl07}). The main purpose of this short paper is to observe that recent results in topological dynamics provide an easy negative answer under additional assumptions.

The proof is based on a simple description of two classical flows (i.e. continuous actions on compact Hausdorff spaces) attached to any (Hausdorff) topological group $G$, which we will describe now. Given a flow $G\actson X$ and $x,y\in X$, the ordered pair $(x,y)$ is \emph{proximal} when there exists a net $(g_{\alpha})_\alpha$ of elements of $G$ such that $\lim_\alpha g_\alpha \cdot x = \lim_\alpha g_\alpha \cdot y$. It is \emph{distal} if $x = y$ or $(x,y)$ is not proximal. The flow $X$ is then \emph{proximal} (resp. distal) when every $(x,y)\in X^{2}$ is proximal (resp. distal). Particular cases of distal flows are provided by \d{equicontinuous} flows, which are those for which for every $x\in X$ and $\epsilon\in \mathcal U_{X}$ (the uniformity of $X$), there exists a neighborhood $U$ of $x$ so that $(g\cdot x, g\cdot y)\in \epsilon$ for every $g\in G$ and $y\in U$. 

By classical results, every topological group $G$ admits a unique universal minimal object $M(G)$ within the class of all flows (that is, a minimal flow that maps onto every other minimal flow), and the same holds in restriction to the classes of all distal, equicontinuous and proximal flows. For these two latter classes, the universal minimal flows are denoted by $B(G)$ and $\Pi(G)$ respectively. It may happen that these flows trivialize, in the sense that they reduce to a single point: this happens for $M(G)$ exactly when $G$ is extremely amenable, and for $B(G)$ exactly when $G$ is minimally almost periodic. Clearly, every extremely amenable group is minimally almost periodic, so Glasner's problem really asks whether the converse holds for monothetic groups. Now, every such group being abelian, it is also \emph{strongly amenable} in the sense that $\Pi(G)$ trivializes (see \cite{Glasner1976}*{II.3.4}). Therefore, the following result provides an answer in the case where $M(G)$ is metrizable: 

\begin{thm}
\label{cor}
Let $G$ be a Polish group and assume that $M(G)$ is metrizable. Suppose that $G$ is strongly amenable and minimally almost periodic. Then $G$ is extremely amenable. 
\end{thm}

Polish groups with metrizable universal minimal flows have been at the center of several recent developments in topological dynamics due to their connection with Ramsey theory. For example, building on the seminal work of Kechris, Pestov and Todorcevic \cite{Kechris05} and its extension \cite{NVT13a}, universal minimal flows and their proximal analogues have been described in \cite{Melleray16} when $G$ is Polish and $M(G)$ metrizable with a generic orbit. Combining the corresponding result with those of  \cite{BenYaacov17} by Ben Yaacov, Melleray and Tsankov (which itself builds on \cite{Zucker16} by Zucker) leads to the following theorem: 

\begin{thm}[\cite{Melleray16}, \cite{BenYaacov17}]
\label{c:SA1}
Let $G$ be a Polish group with metrizable universal minimal flow $G\actson M(G)$. Then there exists a closed, co-precompact, extremely amenable, subgroup $G^{*}$ of $G$ such that $M(G) = \widehat{G/G^{*}}$. In addition, there exists a closed, co-precompact, strongly amenable, subgroup $G^{**}$ of $G$ such that $\Pi(G) = \widehat{G/G^{**}}$; namely, $G^{**}=N(G^{*})$, the normalizer of $G^{*}$ in $G$. In particular, $G$ is strongly amenable iff $G^{*}$ is normal in $G$ iff $M(G)$ is a compact group iff $M(G)$ is distal iff $M(G)$ is equicontinuous.
\end{thm}

The first observation at the origin of the present paper is that a result of the same flavor holds for the distal and the equicontinuous universal minimal flows: 

\begin{thm}
\label{th:UMDF}

Let $G$ be a Polish group with metrizable universal minimal flow. Then the universal minimal distal flow coincides with $B(G)$, and $B(G)=G/H$ for some closed, co-compact subgroup $H$ of $G$. More precisely, writing $M(G) = \widehat{G/G^{*}}$ with $G^{*}$ a closed, co-precompact, extremely amenable subgroup of $G$, one can take $H=(G^{*})^{G}$, where $(G^{*})^{G}$ denotes the closed normal closure of $G^{*}$ in $G$, i.e. $$\displaystyle (G^{*})^{G} = \overline{\langle\bigcup_{g\in G} gG^{*}g^{-1}\rangle}$$

\end{thm}

The second observation is that the combination of the two previous results, yields a direct proof of Theorem \ref{cor}: By Theorem \ref{c:SA1}, $M(G) = \widehat{G/G^{*}}$ with $G^{*}$ a closed, co-precompact subgroup of $G$, and because $G$ is strongly amenable, $G^{*}$ is normal in $G$. Therefore, $(G^{*})^{G}=G^{*}$ and by Theorem \ref{th:UMDF}, $B(G)=G/G^{*}$. But $G$ is minimally almost periodic, so $G^{*}=G$, as required.

The paper is organized as follows: Theorem \ref{th:UMDF} is proved and discussed in Section \ref{section:Proof}. The proof is completely elementary. In Section \ref{section:Examples}, it is used to provide an explicit description of the Bohr compactifications of all those groups $G \leq S_\infty$ that are given as automorphism groups of homogeneous graphs and tournaments. Some familiarity with Fra\"iss\'e theory and with \cite{Kechris05} is assumed. Finally, an open question is presented in Section \ref{section:Questions}.

\section{About Theorem \ref{th:UMDF}}

\label{section:Proof}

\subsection{Proof of Theorem \ref{th:UMDF}}

The proof of Theorem \ref{th:UMDF} rests on some well-known facts about the universal minimal distal and equicontinuous flows, which we shortly recall for completeness. Following \cite{deVries93}, these objects can be described in terms of enveloping semigroups of $G$. Recall first that a \emph{compact right topological semigroup} is a compact Hausdorff space $S$ together with a associative binary operation $*$ so that for every $t\in S$, the map $s\mapsto s*t$ is continuous from $S$ to $S$. An \emph{enveloping semigroup} for a topological group $G$ is a compact right topological semigroup $S$ together with a continuous (not necessarily injective) map $\phi : G\rightarrow S$ so that a) $\phi$ is a homomorphism of semigroups, b) $\phi$ has dense image and c) the map from $G\times S$ to $S$ defined by $(g,s)\mapsto \phi(g)*s$ is continuous. When, in addition $(S,*)$ is a group, it is a \emph{group-like compactification} of $G$. If one further assumes that $(S,*)$ is a topological group, it is a \emph{group compactification} of $G$. Among all group compactifications of $G$, there is a universal one, called the \emph{Bohr compactification} of $G$ and denoted by $\phi_{B}:G\rightarrow \phi_{B}(G)$. It has the following property: for every group compactification $\phi:G\rightarrow K$, there is a continuous homomorphism $\psi$ from $\phi_{B}(G)$ onto $K$ so that $\phi=\psi\circ\phi_{B}$. The homomorphism $\phi_{B}$ allows to see $\phi_{B}(G)$ as a $G$-flow, and $G\actson \phi_{B}(G)$ turns out to be the universal minimal equicontinuous flow of $G$, which we already denoted $G\actson B(G)$ (see \cite{Glasner1976}*{Chapter VIII}). For distal flows, the situation is similar, except that one considers group-like compactifications of $G$ instead of group compactifications (see \cite{deVries93}*{Chapter IV, Section 6.18}). 

Let us now turn to the proof of Theorem \ref{th:UMDF}. The fact that the universal minimal distal and equicontinuous flows of $G$ coincide is a consequence of a general fact in topological dynamics  (see \cite{deVries93}*{IV(6.18-6.19)}): any regular distal minimal flow is equicontinuous whenever it is metrizable. (Recall that $X$ is regular when for every almost periodic point $(x,y)$ in $X^{2}$, there is an endomorphism $\gamma$ of $X$ such that $y=\gamma(x)$.) Here, the universal minimal distal flow is always regular, and it is metrizable because $M(G)$ is.

Next, let $\phi_{B}:G\rightarrow B(G)$ denote the Bohr compactification of $G$. Because it is a minimal $G$-flow, it is a factor of $M(G) = \widehat{G/G^{*}}$ via a map $\pi$. Write $y_{0}$ for $G^{*}$, seen as an element of $M(G)$. It is $G^{*}$-fixed in $M(G)$, so $\pi(y_{0})$ is $G^{*}$-fixed in $B(G)$, and $\pi(y_{0})=\phi_{B}(g^{*})\pi(y_{0})$ for every $g^{*}\in G^{*}$. Therefore, $G^{*}\subseteq\mathrm{Ker} \phi_{B}$, and $(G^{*})^{G} \subseteq \mathrm{Ker} \phi_{B}$. As a result, $\phi_{B}:G\rightarrow B(G)$ induces a continuous morphism $\phi_{B}:G/(G^{*})^{G}\rightarrow B(G)$ with dense image. Notice that since $G^{*}$ is coprecompact in $G$, so is $(G^{*})^{G}$, and the Polish group $G/(G^{*})^{G}$ is in fact compact. Thus, $\phi_{B}$ is surjective, and witnesses that $B(G)$ is a continuous image of $G/(G^{*})^{G}$. As this latter group is a group compactification of $G$, it follows from the definition of the Bohr compactification of $G$ that $B(G)=G/(G^{*})^{G}$. $\qed$

\subsection{Comments on Theorem \ref{th:UMDF}}

\label{subsection:comments} 

Several comments are in order when comparing the statements of Theorem \ref{c:SA1} and Theorem \ref{th:UMDF}. 

First, one may wonder whether $B(G)$ can be shown to be a continuous image of $G$ without assuming $M(G)$ being metrizable. This turns out to be the case: By Ben Yaacov's work \cite{BenYaacov16a}, assuming that $G$ is Roelcke precompact is actually enough for this. The group $G$ is then of the form $\Aut(\m F)$ for some metric $\omega$-categorical Fra\"iss\'e structure $\m F$, and $H$ coincides with the automorphism group of the structure $\m F^{*}$ obtained from $\m F$ by naming all the elements of $\mathrm{acl}(\emptyset)$ in $\m F^{eq}$. Note that when $G$ is non-Archimedean, this is also a consequence of \cite{Tsankov12} or of \cite{NVT17b}. The situation becomes different if we simply assume that $B(G)$ is metrizable, as pointed out kindly by Todor Tsankov: Consider the countable discrete group $G=\mathrm{SL}(3,\mathbb Z)$. It is known to have Kazhdan's property (T) \cite{Kazhdan1967}, which implies that it has a metrizable Bohr compactification \cite{Wang1975}*{Theorem 2.6}. Next, using the natural projections $\mathbb Z \twoheadrightarrow \mathbb Z/n\mathbb Z$, $G$ is residually finite, and as such has an infinite profinite completion, hence an infinite Bohr compactification. Being a compact group, $B(G)$ must therefore be uncountable. Assume now that $B(G)$ is of the form $\widehat{G/H}$. Since $G$ is discrete, so is the uniform structure on the quotient $G/H$ and $B(G)=G/H$ is countable, a contradiction. 

Second, one may ask whether the group $H$ in Theorem \ref{th:UMDF} is minimally almost periodic. This is unclear in general, but holds when $G$ is Roelcke precompact, again in virtue of the results from \cite{BenYaacov16a}.   

Last, let us point out that $G$ may be minimally almost periodic without $M(G)$ being necessarily proximal. For example, gathering results from \cite{Kechris05}, \cite{NVT13a}, \cite{Melleray16} and Section \ref{section:Examples} below, for $\m F=\N$, the random graph, a Henson graph, the random tournament or the rational Urysohn space, $\Aut(\m F)$ is minimally almost periodic,  the universal minimal flow of $\Aut(\m F)$ is the logic on the space of all linear orders on $\m F$, while the proximal universal minimal flow is the logic action on the space of all betweenness relations of $\m F$. For $\m S (2)$, the automorphism group is also minimally almost periodic, the universal minimal flow is the orbit closure of the ``natural'' partition into two halves, and the proximal universal minimal flow is the orbit closure of the corresponding equivalence relation.

\section{Examples of universal minimal distal and equicontinuous flows}

\label{section:Examples}

In this section, we use Theorem \ref{th:UMDF} to calculate the universal minimal equicontinuous flows for the groups $G \leq S_\infty$ that are given as automorphism groups of homogeneous graphs and tournaments. Note that our interest here is really to gather a small catalogue of simple applications of Theorem \ref{cor}, as opposed to prove new results. Indeed, several groups among those considered below are already known to have a simple (in the abstract group-theoretic sense) automorphism group. As a result, the Bohr compactification is trivial. This is so for the random graph by a result of Truss \cite{Truss1985}, and for the Henson graphs and the random tournament, by some unpublished work of Rubin. (The interested reader may consult \cite{Macpherson11} for several specific references.) On the other hand, in the Roelcke precompact case, by the aforementioned result of Ben Yaacov from \cite{BenYaacov16a}, the Bohr compactification can also be obtained by determining $\mathrm{acl}(\emptyset)$ in $\m F^{eq}$, a task which can apparently be carried out without any substantial obstruction in the present case, but may turn out to be difficult in general. For all the arguments that follow, some familiarity with Fra\"iss\'e theory and with \cite{Kechris05} is assumed. 

\subsection{Betweenness relations and minimally almost periodic groups}

\label{subsection:betweenness}

\begin{lemma}
\label{l:B}
  Assume that $\m F$ is a homogeneous structure and that there is an order expansion $\m F^* = (\m F, <)$ so that $M(G) = \widehat{G/G^{*}}$, where $G=\Aut(\m F)$ and $G^{*}=\Aut(\m F^*)$. Assume also that there exist $u, v\in \m F^*$ such that $u < v$ and
for every $x < y \in \m F^*$, there are $x_{0},...,x_{n+1} \in \m F^*$ such that $x_{0}=x$, $x_{n+1}=y$ and $$\exists g_{0}, \ldots, g_n \in G^{*} \ \ \forall i\leq n \quad g_{i}(u)=x_{i} \And g_{i}(v)=x_{i+1}$$

Then $N(G^{*})=\Aut(\m F, B)$, where $B$ is the betweenness relation induced by $<$. 
\end{lemma}

\begin{proof}
See \cite{Melleray16}*{Lemma 5.1}. 
\end{proof}

\begin{lemma}
\label{l:SA}
Assume that $\m F$ is a homogeneous structure and that there is an order Fra\"iss\'e expansion $\m F^* = (\m F, <)$ so that $M(G) = \widehat{G/G^{*}} =\LO(\m F)$, where $G=\Aut(\m F)$ and $G^{*}=\Aut(\m F^*)$. Assume also that there exist $u, v\in \m F^*$ such that $u < v$ and
for every $x < y \in \m F^*$, there are $x_{0},...,x_{n+1} \in \m F^*$ such that $x_{0}=x$, $x_{n+1}=y$ and $$\exists g_{0}, \ldots, g_n \in G^{*} \ \ \forall i\leq n \quad g_{i}(u)=x_{i} \And g_{i}(v)=x_{i+1}$$

Then $\Age(\m F)$ has the strong amalgamation property. 
\end{lemma}

\begin{proof}
First, $M(G)=LO(\m F)$ is equivalent to the fact that $\Age(\m F^{*})$ is the class of all those structures $(\m A, <^{\m A})$ where $\m A\in \Age(\m F)$ and $<^{\m A}$ is a linear ordering on $\m A$. Therefore, the strong amalgamation property of $\Age(\m F)$ is equivalent to the amalgamation property of $\Age(\m F^{*})$ (see \cite{Kechris05}*{Proposition 5.3}). Next, extreme amenability of $G^{*}$ implies that $\Age(\m F^{*})$ has the Ramsey property, which in turn implies that $\Age(\m F^{*})$ has the amalgamation property because it is made of rigid elements, and has the hereditary and the joint embedding properties (\cite{Nesetril1977}*{p.294, Lemma 1}). 
\end{proof}

\begin{lemma}
\label{l:disjoint}

Let $\m F$ be a Fra\"iss\'e structure whose age has the strong amalgamation property. Let $h\in \Aut(\m F)$ with finitely many fixed points, and $\m A$ be a finite substructure of $\m F$. Then there exists a copy $\widetilde{\m A}$ of $\m A$ in $\m F$ so that $\widetilde A \cap h(\widetilde A)=\emptyset$. 

\end{lemma}

\begin{proof}
We proceed by induction on $|\m A|$. The case $|\m A|=1$ is handled thanks to the finiteness of the set of $h$-fixed points, and to the fact that every $1$-point substructure of $\m F$ has infinitely many copies in $\m F$ (thanks to the strong amalgamation property). For the induction step, assume that $|A|=n+1$. Take an enumeration $\{ a_{1},\ldots,a_{n+1}\}$ of $\m A$ and consider $\m A'$ the substructure supported by $\{ a_{1},\ldots,a_{n}\}$. By induction hypothesis, we can find a copy $\widetilde{\m A'}=\{ \tilde a_{1},\ldots,\tilde a_{n}\}$ of $\m A'$ in $\m F$ so that $\widetilde{\m A'}\cap h(\widetilde{\m A'})=\emptyset$. Thanks to the hypotheses on $h$ and of strong amalgamation, we can find $x\in \m F$ so that $h(x)\notin\{ x\}\cup h(\widetilde{A'})\cup h^{-1}(\widetilde{A'})$ and $\widetilde{\m A'}\cup\{x\}\cong \m A$ via $a_{i}\mapsto \tilde a_{i}$ and $a_{n+1}\mapsto x$. Then $\widetilde{\m A}:=\widetilde{\m A'}\cup\{x\}$ is as required. \end{proof}

\begin{lemma}
\label{l:pr}
Suppose that $\m F$ and $\m F^* = (\m F, <)$ are Fra\"iss\'e structures that satisfy the hypothesis of Lemma~\ref{l:SA}. Let $\m A_{0}$ and $\m A_{1}$ be finite disjoint isomorphic substructures of $\m F$. Then there exists $k\in G$ that preserves $<$ on $A_{0}$ and reverses it on $A_{1}$.  
\end{lemma}

\begin{proof}
Consider the substructure $\m B$ of $\m F$ supported by  $A_{0}\cup A_{1}$, together with the ordering $<$ that $\m F^{*}$ induces on it. Define on $\m B$ a new linear ordering $<^{\m B}$ as follows: first, declare $A_{0}<^{\m B} A_{1}$. Next, keep $<$ on $A_{0}$, but reverse it on $A_{1}$. The resulting structure $(\m B, <^{\m B})$ is in $\Age(\m F^{*})$ so it has a copy $\widetilde{\m B}$ in $\m F^{*}$. Furthermore, the identity map from $(\m B, <)$ to $(\m B, <^{\m B})$ is an isomorphism between elements of $\Age(\m F)$. As such, it induces an isomorphism from $\m B$ to $\widetilde{\m B}$ which is order-preserving on $A_{0}$, order-reversing on $A_{1}$, and can be extended to an element $k$ of $G$.
\end{proof}

\begin{prop}
\label{p:map}
Suppose that $\m F$ and $\m F^* = (\m F, <)$ are Fra\"iss\'e structures that satisfy the hypothesis of Lemma~\ref{l:SA}. Then $B(G)$ is trivial. 
\end{prop}

\begin{proof}
Consider $B$ the betweenness relation on $\m F$ induced by $<$. Then $\Aut(\m F,B)$ is the closed subgroup of $G$ generated by $G^{*}$ and any $\sigma\in G$, which we fix from now on, that reverses the ordering. From Lemma \ref{l:B}, this is also the normalizer of $G^{*}$ in $G$. We show that this subgroup is contained in $(G^{*})^{G}$ by showing that $\sigma \in (G^{*})^{G}$. This will suffice to show that $B(G)$ is trivial, because $B(G)=G/(G^{*})^{G}$ will be an equicontinuous factor of $G/N(G^{*})$, which is the universal minimal proximal flow of $G$ by Theorem \ref{c:SA1}. 

To show that $\sigma \in (G^{*})^{G}$, consider $A\subset \m F$ finite and $\m A$ the substructure of $\m F$ supported by $A$. By Lemma \ref{l:SA}, the age of $\m F$ has the strong amalgamation property. Moreover, $\sigma$ has at most one fixed point as it reverses the ordering, so Lemma \ref{l:disjoint} applies and we can find $\widetilde{\m A} \cong \m A$ in $\m F$ so that $\widetilde A\cap \sigma(\widetilde A)=\emptyset$. Applying Lemma \ref{l:pr}, there is $k\in G$ which is order-preserving on $\widetilde A$ and order-reversing on $\sigma(\widetilde A)$. Because $k(\sigma(\widetilde{\m A}))$ and $\m A$ are isomorphic as substructures of $\m F^{*}$, there is $j\in G^{*}$ sending $k(\sigma(\widetilde{\m A}))$ on $\m A$. Set $g = j\circ k$. It is order-preserving from $\widetilde{\m A}$ to $\m A$ and order-reversing on $\sigma(\widetilde{\m A})$. Therefore, the restriction of $g\sigma g^{-1}$ to $A$ is order-preserving.  \end{proof}

Proposition \ref{p:map} allows to capture at once many structures, such as the structure in the empty language, the random graph, all Henson graphs, the random tournament, and the rational Urysohn space (note that the automorphism group of this latter object is not Roelcke-precompact).

\subsection{Homogeneous graphs}

We already computed $B(G)$ in the case of the automorphism groups of the infinite complete graph $K_{\N}$, the Henson graphs and the random graph. According to the Lachlan-Woodrow classification \cite{Lachlan80}, the remaining cases of countable homogeneous graphs are, up to a switch of the edges and the non-edges: 

\begin{enumerate}
\item \label{i:1} $I_{n}[K_{\N}]$, made of $n$ many disjoint copies of $K_{\N}$, where $n\in \N$ is fixed;
\item \label{i:2} $I_{\N}[K_{n}]$, made of infinitely many disjoint copies of $K_{n}$, where $n\in \N$ is fixed;
\item \label{i:3} $I_{\N}[K_{\N}]$, made of infinitely many disjoint copies of $K_{\N}$.
\end{enumerate}

To deal with those, we will use that the relevant groups $G^{*}$ have been described in \cite{Kechris05}, and that the normal closure of $\Aut(\Q,<)$ in $S_{\infty}$ is $S_{\infty}$ itself. 
 
\subsubsection{$I_{n}[K_{\N}]$}
Recall that $G=S_{n}\ltimes S_{\infty}^{n}$ and that $G^{*}=\{e\}\times\Aut(\Q,<)^{n}$. 

Working independently in each part, every element of the normal subgroup $\{ e\} \times S_{\infty}^{n}$ is in $(G^{*})^{G}$. Therefore, $(G^{*})^{G}=\{e\}\times S_{\infty}^{n}=\Aut(I_{n}[K_{\N}], (A^{*}_{i})_{i\in [n]})$ and $B(G)=S_{n}$.  

\subsubsection{$I_{\N}[K_{n}]$}

$G=S_{\infty}\ltimes S_{n} ^{\N}$ and $G^{*}=\Aut(\Q, <)\times \{e\}$. 

First, notice that $S_{\infty}\times \{ e\}\subseteq (G^{*})^{G}$. From this, it is easy to prove that $(G^{*})^{G} = G$, so $B(G)$ is trivial.

\subsubsection{$I_{\N}[K_{\N}]$}

$G=S_{\infty}\ltimes S_{\infty}^\Q$ and $G^{*}=\Aut(\Q, <) \ltimes \Aut(\Q,<)^\Q$. Stabilizing each part setwise, we obtain $\{ e\}\times S_{\infty}^{\Q} \subseteq (G^{*})^{G}$. From this, as before, it is easy to prove that $(G^{*})^{G} = G$ and that $B(G)$ is trivial. 

\subsection{Homogeneous tournaments}

By Lachlan's classification \cite{Lachlan84}, the three countable homogeneous tournaments are $(\Q,<)$, the random tournament, and the dense local order $\m S(2)$. In the first case, the automorphism group is known to be extremely amenable, while the second case follows from the results of Section \ref{subsection:betweenness}. Therefore, the only remaining case to treat is $\m S(2)$. This will be done with the same scheme as for Proposition \ref{p:map}. In what follows, we write $G$ for $\Aut(\m S(2))$. For this structure, it was shown \cite{NVT13a} that $M(G) = \widehat{G/G^*}$, where $G^{*} = \Aut(\m S(2), P^{*} _{0}, P^{*} _{1}) \leq G$ and $P^{*} _{0}, P^{*} _{1}$ is the partition of $\m S(2)$ into right part and left part. Let $E^{*}$ denote the equivalence relation induced by the partition $(P^{*} _{0}, P^{*} _{1})$. We will make use of the following known fact: the structure $(\m S(2), P^{*} _{0}, P^{*} _{1})$ is simply bi-definable with $\Q_{2}=(\Q, <, Q_{0}, Q_{1})$, where both $Q_{0}$ and $Q_{1}$ are dense. To see this, view $(\Q,<)$ as a directed graph where $x\longleftarrow y$ iff $x<y$, and observe that $(\m S(2), P^{*} _{0}, P^{*} _{1})$ is obtained from $(\Q, <, Q_{0}, Q_{1})$ by reversing the edges that are between vertices belonging to different parts. In what follows, we will make use of the ordering $<$ as a relation in $\m S(2)^{*}$ without any further indication.

\begin{lemma}
\label{l:N(Aut(S(2)))}
$N(G^{*})=\Aut(\m S(2), E^{*})$.
\end{lemma}
\begin{proof}
Let $x E^{*} y \in \m S(2)$. Let $g\in N(G^{*})$ and $g^{*}\in G^{*}$ so that $g^{*}(x)=y$. Fix $j \in \{0, 1\}$ such that $g(x)\in P^{*} _{j}$. Then because $gg^{*}g^{-1}\in G^{*}$, we have $gg^{*}g^{-1}(g(x))\in P^{*} _{j}$, i.e., $g(y)\in P^{*} _{j}$. In other words, $g(x) E^{*} g(y)$. So $N(G^{*})\subseteq\Aut(\m S(2), E^{*})$. The other inclusion is easy. 
\end{proof}

\begin{lemma}
\label{l:<}
Let $\sigma \in \Aut(\m S(2),E^{*})\smallsetminus \Aut(\m S(2), P^{*} _{0}, P^{*} _{1})$. Let $A$ be a finite subset of $\m S(2)$ and let $\m A$ (resp. $\m A^{*}$) be the finite substructure that it supports in $\m S(2)$ (resp. $\m S(2)^{*}$). Then there exists a copy $\widetilde{\m A^{*}}$ of $\m A^{*}$ in $\m S(2)^{*}$ so that $\widetilde{\m A^{*}} < \sigma(\widetilde{\m A^{*}})$ or $\sigma(\widetilde{\m A^{*}})<\widetilde{\m A^{*}}$. 
\end{lemma}

\begin{proof}
The proof is similar to that of Lemma \ref{l:disjoint} and is by induction on $|\m A^{*}|$. The base case $|\m A^{*}|=1$ is trivial as $\sigma$ has no fixed point. For the induction step, assume that $|\m A^{*}|=n+1$. Take an increasing enumeration $\{ a_{1},\ldots,a_{n+1}\}$ of $\m A^{*}$ and consider $\m A'$ the substructure of $\m A^{*}$ supported by $\{ a_{1},\ldots,a_{n}\}$. By induction hypothesis, we can find a copy $\widetilde{\m A'}=\{ \tilde a_{1},\ldots,\tilde a_{n}\}$ of $\m A'$ in $\m S(2)^{*}$ so that $\widetilde{\m A'}< \sigma(\widetilde{\m A'})$ or $\sigma(\widetilde{\m A'})< \widetilde{\m A'}$. In the first case, find $x\in \m S(2)^{*}$ so that $\tilde a_{n}<x<\sigma(\tilde a_{1})$ and $\widetilde{\m A'}\cup\{x\}\cong \m A$ via $a_{i}\mapsto \tilde a_{i}$ and $a_{n+1}\mapsto x$. This is possible because both $P^{*}_{0}$ and $P^{*}_{1}$ are dense. Then, because $\sigma$ is order-preserving, we have $\sigma(\tilde a_{n})<\sigma(x)$ and $\widetilde{\m A}:=\widetilde{\m A'}\cup\{x\}$ is as required. In the second case, choose $y$ so that $\sigma(\tilde a_{n})<y<\tilde a_{1}$ and $\sigma(\widetilde{\m A'})\cup\{y\}\cong \m A$ via $a_{i}\mapsto \sigma(\tilde a_{i})$ and $a_{n+1}\mapsto y$. Then, $\tilde a_{n}<\sigma^{-1}(y)$ because $\sigma^{-1}$ is order-preserving, and $\widetilde{\m A}:=\widetilde {\m A'}\cup\{\sigma^{-1}(y)\}$ is as required.
\end{proof}

\begin{lemma}
\label{l:pr'}
Let $\m A_{0}$ and $\m A_{1}$ be finite disjoint isomorphic substructures of $\m S(2)$ so that in $\m S(2)^{*}$, $\m A_{0}<\m A_{1}$ or $\m A_{1}<\m A_{0}$. Then there exists $k\in \Aut(S(2))$ that preserves $P_{0}^{*}$ and $P_{1}^{*}$ on $\m A_{0}$ and permutes them on $\m A_{1}$.
\end{lemma}

\begin{proof}
Consider the substructure $\m B$ of $\m S(2)$ supported by $\m A_{0}\cup\m A_{1}$. As a substructure of $\m S(2)^{*}$, it inherits a partition into two parts $P_{0}^{*}$ and $P_{1}^{*}$, and a linear ordering $<$. Define on $\m B$ a new partition with parts $P_{0}^{\m B}, P_{1}^{\m B}$ and a new linear ordering $<^{\m B}$ as follows: set $P_{0}^{\m B}=P_{0}^{*}$ and $P_{1}^{\m B}=P_{1}^{*}$ on $A_{0}$, but set $P_{0}^{\m B}=P_{1}^{*}$ and $P_{1}^{\m B}=P_{0}^{*}$ on $A_{1}$. As for $<^{\m B}$, if $x<y$, set $x<^{\m B}y$ if $x, y$ are both in $A_{0}$ or $A_{1}$; otherwise, set $y<^{\m B}x$. This is still a linear ordering on $\m B$ because $\m A_{0}<\m A_{1}$ or $\m A_{1}<\m A_{0}$. The directed graph constructed from $(B, P_{0}^{\m B}, P_{1}^{\m B}, <^{\m B})$ by reversing the arcs between elements of different parts is still $\m B$: the arcs supported by $A_{0}$ or $A_{1}$ are not affected by the change of label of the parts, and the arcs between these two sets are preserved because the original ordering has been reversed. In other words, $(\m B, P_{0}^{\m B}, P_{1}^{\m B}, <^{\m B})$ is still an expansion of $\m B$ in $\m S(2)^{*}$. As such, it has a copy $\widetilde{\m B}$ in $\m S(2)^{*}$. Furthermore, the identity map from $(\m B, P_{0}^{*}, P_{1}^{*}, <)$ to $(\m B, P_{0}^{\m B}, P_{1}^{\m B}, <^{\m B})$ is an isomorphism between elements of $\Age(\m S(2))$. As such, it induces an isomorphism $k$ from $\m B$ to $\widetilde{\m B}$ which preserves $P_{0}^{*}$ and $P_{1}^{*}$ on $A_{0}$ and permutes them on $A_{1}$. Extending it to some element of $\Aut(S(2))$ finishes the proof.  
\end{proof}

\begin{prop}

$B(\Aut(\m S(2)))$ is trivial.  

\end{prop}

\begin{proof}

We have seen in Lemma \ref{l:N(Aut(S(2)))} that the normalizer of $\Aut(\m S(2), P^{*} _{0}, P^{*} _{1})$ in $\Aut(\m S(2))$ is $\Aut(\m S(2),E^{*})$. We are going to show that this latter group is contained in $(\Aut(S(2)^{*}))^{\Aut(S(2))}$. This will imply that $$B(\Aut(S(2)))=\Aut(S(2))/(\Aut(S(2)^{*}))^{\Aut(S(2))}$$ is trivial as an equicontinuous factor of $\Aut(S(2))/N(\Aut(S(2)^{*}))$, which is the universal minimal proximal flow of $\Aut(S(2))$ by Theorem \ref{c:SA1}. 

To prove $\Aut(\m S(2),E^{*})\subseteq \Aut(\m S(2)^{*})^{\Aut(\m S(2))}$, let $\sigma \in \Aut(\m S(2),E^{*})\smallsetminus \Aut(\m S(2)^{*})$. As $\sigma$ and $\Aut(S(2){*})$ generate $\Aut(\m S(2),E^{*})$, it suffices to show $$\sigma\in \Aut(\m S(2)^{*})^{\Aut(\m S(2))}$$ 

Let $A$ be a finite subset of $\m S(2)$ and let $\m A$ (resp. $\m A^{*}$) be the finite substructure that it supports in $\m S(2)$ (resp. $\m S(2)^{*}$). By Lemma \ref{l:<}, there exists a copy $\widetilde{\m A^{*}}$ of $\m A^{*}$ in $\m S(2)^{*}$ so that $\widetilde{A^{*}} < \sigma(\widetilde A^{*})$ or $\sigma(\widetilde A^{*})<\widetilde A^{*}$. Lemma \ref{l:pr'} applies, and there exists $k\in \Aut(S(2))$ that preserves $P_{0}^{*}$ and $P_{1}^{*}$ on $\widetilde A$ and permutes them on $\sigma(\widetilde A)$. Because $k(\sigma(\widetilde{\m A}))$ and $\m A^{*}$ are isomorphic as substructures of $\m S(2)^{*}$, there is $j\in G^{*}$ sending $k(\sigma(\widetilde{\m A}))$ on $\m A^{*}$. Set $g = j\circ k$. It sends $\widetilde{\m A}$ to $\m A$, preserves $P_{0}^{*}$ and $P_{1}^{*}$ on $\widetilde{\m A}$ and permutes them on $\sigma(\widetilde{\m A})$. Therefore, the restriction of $g\sigma g^{-1}$ to $A$ preserves $P_{0}^{*}$ and $P_{1}^{*}$. \end{proof}

\section{Comments and questions}

\label{section:Questions}

We close this paper with a question.

\begin{question}
Let $G$ be a Polish group such that $\Pi(G)$ and $B(G)$ are metrizable. Is $M(G)$ necessarily metrizable? 
\end{question}

In view of Theorem \ref{cor}, a positive answer would solve Glasner's problem in a rather strong sense. However, let us mention again that such an outcome would go against the intuition of many experts in the field, including Glasner himself. 

\subsection*{Acknowledgements}

I am grateful to Ita\"i Ben Yaacov for his reference to \cite{BenYaacov16a} and his explanations regarding it; to Julien Melleray for the references regarding simplicity of automorphism groups; to Todor Tsankov for making even simpler the original proof of Theorem \ref{th:UMDF}, for his references and explanations regarding \cite{Tsankov12}, and for the comment of Section \ref{subsection:comments} regarding $\mathrm{SL}(3, \mathbb Z)$; and finally to the anonymous referee for her/his suggestions and comments, which substantially improved the clarity of the paper.

\bibliography{Bib18Mar}
\end{document}